\documentclass[12pt]{article}
\usepackage{amsmath,amssymb,amsfonts}
\usepackage{amsthm}
\usepackage{enumerate}
\usepackage[dvipsnames]{xcolor}
\usepackage[pagebackref,colorlinks,citecolor=blue,linkcolor=blue]{hyperref}

\usepackage[top=1in, bottom=1in, left=1.35in, right=1.35in]{geometry}
\def\Xint#1{\mathchoice
{\XXint\displaystyle\textstyle{#1}}%
{\XXint\textstyle\scriptstyle{#1}}%
{\XXint\scriptstyle\scriptscriptstyle{#1}}%
{\XXint\scriptscriptstyle\scriptscriptstyle{#1}}%
\!\int}
\def\XXint#1#2#3{{\setbox0=\hbox{$#1{#2#3}{\int}$ }
\vcenter{\hbox{$#2#3$ }}\kern-.59\wd0}}
\def\avgint{\Xint-}

\newcommand{\BV}{\mathrm{BV}}
\newcommand{\N}{\mathbb{N}}
\newcommand{\R}{\mathbb{R}}
\newcommand{\Z}{\mathbb{Z}}

\newcommand{\Om}{\Omega}
\newcommand{\eps}{\varepsilon}
\newcommand{\loc}{\mathrm{loc}}

\newcommand{\liploc}{\mathrm{Lip}_{\mathrm{loc}}}

\DeclareMathOperator{\rcapa}{cap}
\DeclareMathOperator{\capa}{Cap}
\DeclareMathOperator{\rad}{rad}
\DeclareMathOperator{\Mod}{Mod}
\DeclareMathOperator{\diam}{diam}
\DeclareMathOperator{\dist}{dist}
\DeclareMathOperator{\inte}{int}
\DeclareMathOperator{\ext}{ext}
\DeclareMathOperator*{\essinf}{ess\,inf}
\DeclareMathOperator{\Lip}{Lip}

\newcommand{\fin}{f^{-1}}

\newtheorem{theorem}{Theorem}[section]
\newtheorem{proposition}[theorem]{Proposition}

\newtheorem{lemma}[theorem]{Lemma}

\theoremstyle{definition}
\newtheorem{defn}[theorem]{Definition}

\theoremstyle{remark}
\newtheorem{remark}[theorem]{Remark}

\numberwithin{equation}{section}

\begin{document}

\title{Duality of moduli and quasiconformal mappings in
metric spaces}
\author{Rebekah Jones and Panu Lahti\footnote{The authors are grateful to Nageswari
Shanmugalingam for advice and many discussions on the topic of the paper.
The second author wishes to acknowledge the hospitality of the
University of Cincinnati, where most of the research for this paper was conducted.
The research was partially funded
by the National Science Foundation (U.S.A.) grants DMS \#1500440
and DMS \#1800161.
}
}
\maketitle

\noindent{\small
{\bf Abstract.}
We prove a duality relation for the moduli of the family of curves connecting two sets
and the family of surfaces separating the sets, in the setting of a complete metric space
equipped with a doubling measure and supporting a Poincar\'e inequality.
Then we apply this to show that quasiconformal mappings can be characterized
by the fact that they quasi-preserve the modulus of certain families of surfaces.
}

\bigskip


\medskip


\section{Introduction}

A homeomorphism $f\colon X\to Y$ between two metric spaces $X,Y$ is said to be
quasiconformal if
there is a constant $H\ge 1$ such that for all $x\in X$,
\[
\limsup_{r\to 0^+}\frac{\sup_{y\in \overline{B}(x,r)}d_Y(f(x),f(y))}{\inf_{y\in X\setminus B(x,r)}d_Y(f(x),f(y))}\le H.
\]
In metric measure spaces satisfying suitable conditions such
as Ahlfors regularity and a Poincar\'e inequality,
the study of quasiconformal mappings was begun by Heinonen and Koskela
in~\cite{HK} and by now the literature is extensive,
see for example~\cite{BHW, HK0, HKST, KMS, Wil}.
As in the classical Euclidean setting, there are also other notions of quasiconformality.
For Ahlfors $Q$-regular spaces $X,Y$, a homeomorphism $f\colon X\to Y$ is
said to be \emph{geometric quasiconformal} if
there is a constant $K\ge 1$ such that whenever $\Gamma$ is a family
of curves in $X$, we have
\[
\frac{1}{K}\Mod_Q(f\Gamma)\le \Mod_Q(\Gamma)\le K \Mod_Q(f\Gamma),
\]
For the definition of $Q$-modulus and all other concepts needed in the paper,
we refer to Section \ref{sec:definitions}.
If both $X$ and $Y$ are complete and also support a $Q$-Poincar\'e inequality,
the two notions of quasiconformality are equivalent, see Theorem 9.8 in \cite{HKST}.

A fact that has received much less attention is that quasiconformal
mappings also quasi-preserve the
$\tfrac{Q}{Q-1}$-modulus of certain families of surfaces obtained as
``essential boundaries" of sets of finite perimeter.
This result was proved in Euclidean spaces by Kelly~\cite[Theorem~6.6]{Kelly}.
In the metric space setting, the theory of functions of bounded variation
(BV) and sets of finite perimeter
was first developed by Ambrosio and Miranda \cite{A1,M}.
The authors of the current paper together with Shanmugalingam
extended Kelly's result to metric spaces in \cite{JLS}.

In the current paper, our main goal is to show that the converse holds
as well: if a homeomorphism $f$ quasi-preserves the modulus of families of surfaces,
then it is a quasiconformal mapping.
Since the analogous fact is already known to hold for families of curves,
we invest most of our efforts
in studying the \emph{duality} of moduli of families of curves and surfaces.
Specifically, for a nonempty bounded open set $\Om\subset X$ and two
disjoint sets $E,F\subset \Om$, we consider the family of curves $\Gamma$ joining $E$ and $F$ in $\Om$,
and the family of surfaces $\mathcal L$ separating $E$ and $F$ in $\Om$ in a suitable sense.
Then we prove the following theorem; the precise formulation and assumptions
on the sets $E$ and $F$ are given in Theorem \ref{thm:duality of moduli in text}.
\begin{theorem}\label{thm:duality of moduli}
Let $1<p<\infty$ and suppose $X$ is a complete metric space
equipped with a doubling
measure and supporting a $1$-Poincar\'e inequality.
For some constant $C\ge 1$ depending only on $p$ and the space $X$,
we have
\[
1 \le \Mod_{\frac p{p-1}}(\mathcal{L})^{\frac{p-1}p}\Mod_p(\Gamma)^{\frac 1p}\le C.
\]
\end{theorem}
In Euclidean spaces, this was proved (with constant $C=1$)
by Ziemer \cite{Z}, and later
by Aikawa and Ohtsuka who show in \cite{AO} that
the same result	holds for a more general weighted modulus with weights coming from the Muckenhoupt $A_p$-class.
Combining Theorem \ref{thm:duality of moduli} with the characterization of quasiconformal mappings
by means of the moduli of curve families, we get the following theorem.

\begin{theorem}\label{thm:quasiconformality}
Suppose that $X$ and $Y$ are complete Ahlfors $p$-regular metric spaces
supporting a $1$-Poincar\'e inequality.
Suppose $f\colon X\to Y$ is a homeomorphism and there exists $C_0>0$
such that for every collection $\mathcal L$ of surfaces in $X$, 
\[
\Mod_{\frac p{p-1}}(f\mathcal L)\le C_0 \Mod_{\frac p{p-1}}(\mathcal L).
\]
Then $f$ is quasiconformal.
\end{theorem}

This is given, in a somewhat more general form, in Theorem
\ref{thm:quasiconformality in text}.
Results similar to Theorem \ref{thm:duality of moduli}
and Theorem \ref{thm:quasiconformality} were very recently proved
in the metric space setting by Lohvansuu and Rajala \cite{LR}, but their
viewpoint was
somewhat different. In \cite{LR}
(similarly to \cite{Z}) the authors understood a ``surface'' to be
a set of finite codimension one Hausdorff measure separating $E$
and $F$ in a topological sense. By contrast, we understand surfaces to be sets of finite 
perimeter in the spirit of \cite{Kelly} and \cite{JLS}.

Moreover, we wish to study the problem under weaker assumptions:
instead of Ahlfors regularity it is in fact enough to assume
in Theorem \ref{thm:quasiconformality} that the measures
on $X$ and $Y$ are doubling and satisfy suitable one-sided growth bounds.
Additionally, we do not assume the sets $E$ and $F$ to be closed, as was done
in \cite{LR} and \cite{Z}. Working with more general sets makes it a rather
subtle problem to find the correct definition for a ``surface'' that separates $E$ and $F$;
for this we apply the concept of \emph{fine topology},
relying on results proved in \cite{BB,BBL,JB}.
Hence our arguments combine the theory of quasiconformal mappings,
BV theory, and fine potential theory in metric spaces.

\section{Notation and definitions}\label{sec:definitions}

In this section we gather the definitions and assumptions that we need in the paper.

Throughout the paper,
$(X,d,\mu)$ is a complete metric measure space with $\mu$ a Radon measure.
We assume that $X$ consists of at least 2 points.
If a property holds outside a set with $\mu$-measure zero, we say that it holds
almost everywhere, or a.e.

Given $x\in X$ and $r>0$, we denote an open ball by
$B(x,r):=\{y\in X:\,d(y,x)<r\}$. 
Given that in a metric space a ball, as a set, could have more than one
radius and more than one
center, we will consider a ball to be also equipped with a radius and center;
thus two different balls might
correspond to the same set. We then denote
$\rad (B):=r$ as the pre-assigned radius of the ball $B=B(x,r)$,
and $aB:=B(x,ar)$ for $a>0$.

\begin{defn}\label{def:doubling and Ahlfors regularity}
We say that $\mu$ is \emph{doubling} if there exists a constant
$C_d\ge 1$, called the \emph{doubling constant}, such that
\[
0<\mu(2B)\le C_d \mu(B)<\infty
\]
for every ball $B$.

We say that
$(X,d,\mu)$ is \emph{Ahlfors $Q$-regular}, with $Q>0$, if
there is
a constant $C_A\ge 1$ 
such that whenever $x\in X$ and $0<r<\diam(X)$, we have
\[
\frac{r^Q}{C_A}\le \mu(B(x,r))\le C_A\, r^Q.
\]
\end{defn}

Throughout the paper, we always assume $\mu$ to be doubling.

\begin{defn}
Let $A\subset X$. The \emph{codimension 1 Hausdorff measure}
of $A$ is given by 
\[
\mathcal{H}(A):=\lim_{r\to 0^+}\inf\left\{\sum_{k\in I}\frac{\mu(B_k)}{\rad(B_k)}\,\middle\vert\,A\subset\bigcup_{k\in I}B_k
\text{ where }\rad(B_k)\le r\text{ and } I\subset\N\right\}.
\]
\end{defn}

Note that a complete metric space equipped with a doubling measure is always proper,
that is, closed and bounded sets are compact.
Given an open set $\Om\subset X$,
we write $u\in L^1_{\loc}(\Om)$ if $u\in L^1(V)$ for every open
$V\Subset \Om$; this expression means that $\overline{V}$ is a compact
subset of $\Om$.
Other local spaces are defined analogously.

A \emph{curve} is a continuous mapping from a compact interval into
$X$, and a \emph{rectifiable} curve is a curve with finite length.
The length of a rectifiable curve $\gamma$
is denoted by $\ell_{\gamma}$. Every rectifiable curve can be parametrized
by arc-length, see e.g. \cite[Theorem~3.2]{Hj}.
In the following definitions, we let $1\le p<\infty$; in most of the paper
we will assume that $1<p<\infty$.

\begin{defn}\label{def:modulus}
Let $\mathcal{M}$ be a collection of Borel measures on $X$. The admissible class of $\mathcal{M}$, 
denoted $\mathcal{A}(\mathcal{M})$, is the set of all nonnegative Borel functions
$\rho\colon X\to [0,\infty]$ such that
\[ 
\int_X \rho \ d\lambda \ge 1 
\]
for all $\lambda \in \mathcal{M}$. The \emph{$p$-modulus} of the family $\mathcal{M}$ is given by
\[
\Mod_p(\mathcal{M}):=\inf_{\rho \in \mathcal{A}(\mathcal{M})} \int_X \rho^p \ d\mu. 
\]
We say that a nonnegative Borel
function $\rho$ is $p$-\emph{weakly admissible} for the collection $\mathcal M$ if $\rho$ is 
admissible for all but a $p$-modulus zero collection of measures.
\end{defn}

$\Mod_p$ is an outer measure on the class of all Borel measures, see~\cite{F}.
There are two types of collections of measures associated with quasiconformal
mappings. Firstly, given
a collection $\Gamma$ of curves in $X$, we set $\Gamma$ to also denote the  
arc-length measures restricted to each curve in $\Gamma$; then the admissibility
condition is replaced by
\[
\int_{\gamma}\rho\,ds\ge 1
\]
for every rectifiable $\gamma\in\Gamma$, where
\[
\int_{\gamma}\rho\,ds:=\int_0^{\ell_{\gamma}}\rho(\gamma(s))\,ds
\]
for rectifiable $\gamma$.
We  say that a property holds for $p$-almost every curve if it
fails only for a curve family with zero $p$-modulus.
Secondly, for a collection $\mathcal L$ of sets of finite perimeter in a set $\Om$, 
we consider the measures $P(U,\cdot)$
for each $U\in\mathcal L$ (see the definition given later).

\begin{defn}\label{def:N1p}
	Let $\Om\subset X$ be $\mu$-measurable.
Given a function $u\colon \Om\to \overline{\R}$, a Borel function
$g\colon \Om\to[0,\infty]$ is said to be an \emph{upper gradient of $u$ in $\Om$} 
if for every nonconstant rectifiable curve $\gamma$ in $\Om$,
\begin{equation}\label{eq:upper gradient inequality}
|u(x)-u(y)|\le\int_\gamma g\, ds,
\end{equation}
where $x$ and $y$ are the endpoints of $\gamma$.
We interpret $|u(x)-u(y)|=\infty$ whenever either $|u(x)|$ or $|u(y)|$ is infinite.
A function $u$ is said to be in
the \emph{Newton-Sobolev class} $N^{1,p}(\Om)$
if $u\in L^p(\Om)$ and there is an upper gradient $g$ of $u$ in $\Om$ such that
$g\in L^p(\Om)$.
We let
\[
\Vert u\Vert_{N^{1,p}(\Om)}:=\Vert u\Vert_{L^p(\Om)}+\inf \Vert g\Vert_{L^p(\Om)},
\]
where the infimum is taken over upper gradients $g$ of $u$ in $\Om$.
We say that a nonnegative $\mu$-measurable function $g$ is a $p$-weak upper gradient
of a function $u$ in $\Om$ if \eqref{eq:upper gradient inequality} holds
for $p$-almost every curve in $\Om$.
\end{defn}

If $u\in N_{\loc}^{1,p}(\Om)$, then there exists a \emph{minimal} $p$-weak
upper gradient of $u$ in $\Om$, always denoted by $g_u$,
satisfying $g_u\le g$ a.e. in $\Om$ for
every $p$-weak upper gradient $g\in L_{\loc}^p(\Om)$ of $u$ in $\Om$; see
\cite[Theorem 2.25]{BB}
We refer the reader to~\cite{BB,HKST15,S} for more details regarding mappings in
$N^{1,p}_{\loc}(\Om)$.

\begin{defn}
We say that the space $X$ supports a \emph{$p$-Poincar\'e inequality} 
if there exist constants $C_P>0$ and $\lambda\ge 1$ 
such that for all balls $B$ in $X$, all measurable functions $u$ on $X$ and 
all upper gradients $g$ of $u$,
\[
\avgint_{B}|u-u_B| \,d\mu \le C_P\rad(B)\left(\avgint_{\lambda B}g^p \,d\mu\right)^{1/p}.
\]
Here we denote the integral average of $u$ over $B$ by
\[
u_B:=\avgint_{B}u\,d\mu:=\frac{1}{\mu(B)}\int_B u\,d\mu.
\]
\end{defn}

We will assume throughout the paper that $X$ supports a $1$-Poincar\'e inequality.

\begin{defn}\label{def:Loewner space}
For any disjoint sets $E,F\subset X$, we
define $\Gamma(E,F;X)$ to be the collection of curves in $X$ joining
$E$ and $F$.
We say that $X$ is a \emph{Loewner space} if there is a function
$\phi\colon (0,\infty)\to (0,\infty)$ such that
\[
\Mod_p(\Gamma(E,F;X))\ge \phi(t)
\]
whenever $E$ and $F$ are two disjoint, nondegenerate continua (compact connected sets)
such that
\[
t\ge \Delta(E,F):=\frac{\dist(E,F)}{\min\{\diam (E),\diam (F)\}}.
\]
\end{defn}

\begin{defn}
The \emph{$p$-capacity} of a set $A\subset X$ is given by
\[
\capa_p(A):=\inf \Vert u\Vert_{N^{1,p}(X)},
\]
where the infimum is taken over test functions satisfying $u\ge 1$ in $A$.
If a property holds outside a set with $p$-capacity zero, we say that it holds
\emph{p-quasieverywhere}, or $p$-q.e.

We say that a set $V\subset X$ is \emph{p-quasiopen} if for every $\eps>0$ there is
an open set $G\subset X$ such that $\capa_p(G)<\eps$ and
$V\cup G$ is open.

The \emph{relative $p$-capacity} of two sets $A\subset \Om \subset X$ is given by 
\[
\rcapa_p(A,\Om) := \inf \int_X g_u^p \ d\mu
\]
where the infimum is over all  
functions $u\in N^{1,p}(X)$ such that $u\ge 1$ $p$-q.e. in $A$
and $u\le 0$ in $X\setminus \Om$. Recall that $g_u$ denotes the minimal
$p$-weak upper gradient of $u$.
\end{defn}

We know that $\capa_p$ is an outer capacity in the following sense:
\[
\capa_p(A) = \inf \{\capa_p(W) : W \supset A,\ W \ \textrm{is open}\}
\]
for any $A \subset X$, see e.g. \cite[Theorem 5.31]{BB}.

If $\Om\subset X$ is $\mu$-measurable, then 
\begin{equation}\label{eq:quasieverywhere equivalence class}
v=0\ \textrm{ p-q.e. in }\Om\textrm{ implies }\ \Vert v\Vert_{N^{1,p}(\Om)}=0,
\end{equation}
see \cite[Proposition 1.61]{BB}.

From now on, let $1<p<\infty$.

\begin{defn}\label{def:thinness}
A set $E\subset X$ is \emph{$p$-thin} at $x\in X$ if 
\[
\int_0^1 \left(\frac{\rcapa_p(E\cap B(x,t),B(x,2t))}{\rcapa_p(B(x,t),B(x,2t))}\right)^{\frac{1}{p-1}} \frac {dt}{t} <\infty.
\]
If $E$ is not $p$-thin at $x$, we say that it is $p$-thick. We denote the collection
of all points where $E$ is $p$-thick by $b_p E$. If $X\setminus E$ is $p$-thin at each
point $x\in E$, we say that the set $E$ is $p$-\emph{finely open}.
Then the $p$-fine topology on $X$ is the collection of all $p$-finely
open sets.
\end{defn}

\begin{defn}\label{def:condenser}
Given a nonempty open set $\Om$ and two disjoint sets $E,F$, we define the capacity of the \emph{condenser} $(E,F;\Om)$ by
\[
\rcapa_p(E,F;\Om):=\inf\int_{\Om}g_u^p\,d\mu,
\]
where the infimum is taken over all $u\in N^{1,p}(\Om)$ satisfying
$0\le u\le 1$ in $\Om$, $u=1$ in $E\cap \Om$, and $u=0$ in $F\cap \Om$.
\end{defn}

\begin{defn}\label{def:superminimizers}
A function $u\in N_{\loc}^{1,p}(\Om)$ is a $p$-\emph{minimizer} in an open set $\Om\subset X$
if for all $\varphi\in \Lip_c(\Om)$ we have
\[
\int_{\{\varphi\neq 0\}}g_u^p\,d\mu\le \int_{\{\varphi\neq 0\}} g_{u+\varphi}^p\,d\mu.
\]
If the above holds for all nonnegative $\varphi\in \Lip_c(\Om)$, we say that
$u$ is a $p$-\emph{superminimizer}, and if it holds for all nonpositive $\varphi\in \Lip_c(\Om)$, 
we say that $u$ is a $p$-\emph{subminimizer}.
\end{defn}

Next we consider the theory of BV functions in metric spaces.

\begin{defn}
For an open set $\Om\subset X$ and $u \in L^1_{\loc} (\Om)$,
the total variation of $u$ in $\Om$ is given by
\[
\Vert Du\Vert (\Om):=\inf\left\{\liminf_{n\to\infty}\int_\Om g_{u_n}\, d\mu:\,
		(u_n)_{n\in\N}\subset\liploc(\Om),\,u_n\to u\text{ in } L^1_{\loc}(\Om)\right\}.
\]
We say $u\in L^1(\Om)$ is of \emph{bounded variation} on $\Om$, denoted
$u\in \BV(\Om)$, if $\Vert Du\Vert (\Om)<\infty$.
\end{defn}

It is shown in \cite[Theorem 3.4]{M} that $\Vert Du\Vert$ is a Radon measure in $\Om$
for any $u\in \BV_{\loc}(\Om)$. We call 
$\Vert Du\Vert $ the variation measure of $u$. 

\begin{defn}\label{def:finite-perimeter}
A measurable set $U\subset X$ has \emph{finite perimeter} in $\Om$ if
$\Vert D\chi_U\Vert(\Om)<\infty$. We call 
$\Vert D\chi_U\Vert$ the \emph{perimeter measure} of $U$ and we will denote it $P(U,\cdot)$.
\end{defn}

\begin{defn}\label{def:relative isoperimetric inequality}
We say that $X$ supports a \emph{relative isoperimetric inequality} if there exist constants $C_I>0$ and 
$\lambda\ge 1$ such that for all balls $B$ and for all
measurable sets $U$, we have 
\[
\min\{\mu(B\cap U),\,\mu(B\setminus U)\}\le C_I \rad(B) P(U,\lambda B).
\]
\end{defn}

We know that when $\mu$ is doubling and $X$ supports a
$1$-Poincar\'e inequality, then it supports a relative isoperimetric inequality,
see for example \cite[Theorem 3.3]{KoLa} (in a slightly different form,
this was proved earlier in \cite[Theorem 4.3]{A1}).

The noncentered Hardy-Littlewood maximal function of a function
$\rho\in L^1_{\loc}(X)$
is defined by
\begin{equation}\label{eq:definition of maximal function}
\mathcal M \rho(x):=\sup_{B\ni x}\,\avgint_{B}|\rho|\,d\mu,
\end{equation}
where the supremum is taken over all open balls containing $x\in X$.

Finally we give the definition of quasiconformal mappings on metric spaces.
Let $(Y,d_Y,\mu_Y)$ be another metric space equipped with a Radon measure
$\mu_Y$.

\begin{defn}\label{def:dilations}
For a function $f\colon X\to Y$, define for all $x\in X$ and $r>0$
\[ 
L_f(x,r):=\sup_{y\in \overline{B}(x,r)}d_Y(f(x),f(y))
\quad\textrm{and}\quad
l_f (x,r):=\inf_{y\in X\setminus B(x,r)}d_Y(f(x),f(y)).
\]
A homeomorphism $f$ is \emph{(metric) quasiconformal} if
there is a constant $H\ge 1$ such that for all $x\in X$ we have
\[
  \limsup_{r\to 0^+}\frac{L_f(x,r)}{l_f(x,r)}\le H.
\]
A homeomorphism $f$ is \emph{geometric quasiconformal} if
there is a constant $K\ge 1$ such that whenever $\Gamma$ is a family
of curves in $X$, we have
\[
\frac{1}{K}\Mod_p(f\Gamma)\le \Mod_p(\Gamma)\le K \Mod_p(f\Gamma).
\]
\end{defn}

Note that when $f$ is a homeomorphism, we always have
$l_f(x,r)\le L_f(x,r)$.
It is known that when both $X$ and $Y$ are
Ahlfors $p$-regular and support a $p$-Poincar\'e inequality, the two notions
of quasiconformality are equivalent, see Theorem 9.8 in \cite{HKST}.
We will make use of this fact in Section \ref{sec:qc characterization},
but we will give a self-contained proof where we only need somewhat weaker
assumptions than Ahlfors regularity.\\

\noindent {\bf Standing assumptions:} \label{StandAssuption}
Throughout this paper we will assume that $1<p<\infty$ and that
$(X,d,\mu)$ is a complete metric measure space
that supports a $1$-Poincar\'e inequality, such that
$\mu$ is doubling.
We will use the letter $C$ to denote various
nonnegative constants that depend
only on $p$ and the space $X$, and
the value of $C$ could differ at each occurrence.

\section{Background results}

In this section we will gather most of the background 
results needed in the paper. We start with the following coarea formula for BV
functions, which is stated in Remark 4.3 of \cite{M}.

\begin{theorem}\label{thm:coarea}
Suppose $\Om \subset X$ is open and
$u\in \BV(\Om)$. For each $t\in\R$, denote the super-level set of $u$ by
$\{u>t\}:=\{x\in\Om:\,u(x)>t\}$.
Then for every nonnegative Borel function $\phi$ on $\Om$ and every Borel set $A\subset \Om$,
\[
\int_\R\left(\int_A \phi \, dP(\{u>t\},\cdot)\right)dt=\int_A \phi \, d\Vert Du\Vert.
\]
\end{theorem}

We have the following ``continuity from below'' for families of measures;
for a proof see Lemma~\ref{lem:mod-continuity} in \cite{Z2}.

\begin{lemma}\label{lem:mod-continuity}
If $\{\mathcal{L}_j\}_{j\in\N}$ is a sequence of families of Borel measures
such that
$\mathcal{L}_j\subset \mathcal{L}_{j+1}$ for each $j$, then
\[
\Mod_p\Big(\bigcup_{j\in\N}\mathcal{L}_j\Big) = \lim_{j\to\infty}\Mod_p(\mathcal{L}_j).
\]
\end{lemma}

By applying Fuglede's and Mazur's lemmas, see e.g. \cite[p.19, p.131]{HKST},
we get the following.

\begin{lemma}\label{lem:existence of minimal rho}
Let $\mathcal L$ be a family of Borel measures with $\Mod_p(\mathcal L)<\infty$.
Then there exists a $p$-weakly admissible function $\rho$ such that
\[
\int_X \rho^p\,d\mu=\Mod_p(\mathcal L).
\]
\end{lemma}

The following lemma is proved in \cite[Lemma~5.2]{LR}.

\begin{lemma}\label{lem:phi-and-rho}
If $\rho$ is a $p/(p-1)$-weakly admissible function for a family of
Borel measures $\mathcal L$ such that 
\[
\Mod_{\frac p{p-1}}(\mathcal L) = \int_X \rho^{\frac{p}{p-1}}\, d\mu
\]
and $\phi$ is another $p/(p-1)$-weakly admissible function for $\mathcal L$, then 
\[
\Mod_{\frac p{p-1}}(\mathcal L) \le \int_X \phi\, \rho^{\frac 1{p-1}}\, d\mu.
\]
\end{lemma}

We note that various results that we cite, such as the following
theorem, rely on assuming the space to support
a $p$-Poincar\'e inequality, but this follows via H\"older's inequality
from the $1$-Poincar\'e inequality that is our standing assumption.

\begin{theorem}\label{thm:Loewner}
Suppose that $X$ satisfies the lower mass bound
\[
\mu(B(x,r))\ge cr^p
\]
for all $x\in X$ and $0<r<\diam (X)$, and some constant $c>0$. Then $X$ is a Loewner space.
\end{theorem}
\begin{proof}
See Theorem 5.7 in \cite{HK}. Note that the so-called $\varphi$-convexity assumed
in this theorem holds since under our assumptions the space is \emph{quasiconvex},
meaning that every pair of points can be joined by a curve whose length is at
most a constant number times the distance between the points;
see e.g. \cite[Theorem 4.32]{BB}.
\end{proof}

The space $X$ is \emph{linearly locally connected} in the following sense.

\begin{theorem}\label{thm:linear local connectivity}
Suppose $X$ satisfies the upper mass bound
\[
\mu(B(x,r))\le C_0 r^p
\]
for all $x\in X$ and $r>0$, and a constant $C_0>0$.
Then there exists a constant $C_1\ge 1$
such that
for every ball $B=B(x,r)$, any pair of points in
$B\setminus \tfrac 12 B$
can be joined by a curve in $B(x,C_1 r)\setminus B(x,r/C_1)$.
\end{theorem}
\begin{proof}
See Remark 3.19 in \cite{HK}; note that there it is
also assumed that the space is of Hausdorff dimension $p$, but this is not needed
in the proof.
\end{proof}

Finally we give a few results concerning superminimizers;
recall Definition \ref{def:superminimizers}.
Let $W\subset X$ be an open set.
We define the \emph{lsc-regularization} (lower semicontinuous regularization) of
a function $u$ on $W$ by
\[
u^*(x):=\lim_{r\to 0} \essinf_{B(x,r)} u,\quad x\in W.
\]

The following proposition is given as part of Theorem 8.22 in \cite{BB}.

\begin{proposition}\label{prop: facts about lsc regularization}
If $u$ is a $p$-superminimizer in $W$, then $u^*$ is lower
semicontinuous in $W$ and $u=u^*$ $p$-q.e. in $W$.
\end{proposition}

More precisely, the fact that $u=u^*$ $p$-q.e. in $W$ is given in the \emph{proof}
of \cite[Theorem 8.22]{BB}. By \eqref{eq:quasieverywhere equivalence class}
we know that $u^*$ is still a $p$-superminimizer.

It is a well known fact that \emph{superharmonic functions} are \emph{finely continuous};
this was shown in the metric space setting in \cite{JB} and \cite{Korte}.
Here we record this result in the following theorem, which follows by combining
Proposition 7.12, Theorem 9.24(a,c), and Theorem 11.38 of \cite{BB}.
\begin{theorem}\label{thm:fine continuity}
Let $u$ be a $p$-superminimizer in $W$. Then $u^*$ is continuous with respect to
the $p$-fine topology in $W$.
\end{theorem}	

\section{Proof of Theorem \ref{thm:duality of moduli}}

We will consider the following families of curves and surfaces; recall the concept
of capacitary thinness from Definition \ref{def:thinness}.

\begin{defn}
For an open set $\Om\subset X$ and any disjoint sets $E,F\subset X$, we
define $\Gamma(E,F;\Omega)$ to be the collection of curves in $\Omega$ joining
$E\cap\Om$ and $F\cap\Om$.
We also  define the collection of measures

\begin{align*}
&\mathcal L(E,F;\Omega):=
\{P(U,\Om\cap\cdot):\,
U\subset X\textrm{ is }\mu\textrm{-measurable with }\\
&\qquad\qquad\qquad\qquad\quad b_p E\cap \Om\subset \inte(U) \text{ and } b_p F\cap \Om\subset \ext (U)\}.
\end{align*}
\end{defn}
By an abuse of terminology,
we will also talk about the sets $U$ belonging to
$\mathcal L(E,F;\Omega)$.
Essentially, the boundaries of $U$ are ``surfaces'' that ``separate'' $E$ and $F$ in $\Om$,
but since we do not assume $E$ and $F$ to necessarily be compact subsets of $\Om$,
the choice of the correct definition for $\mathcal L(E,F;\Omega)$ becomes rather subtle.
If one would employ the usual definition where the surfaces need to stay at a strictly
positive distance from $E$ and $F$, it would be difficult to prove the
lower bound of Theorem \ref{thm:duality of moduli}. On the other hand,
if one allows the surfaces to ``touch'' $E$ and $F$ significantly, then it becomes
difficult to prove the upper bound. For this reason, we allow the surfaces to
``touch'' $E$ and $F$ only at capacitary thinness points.

Throughout this section, we will abbreviate
$\mathcal L = \mathcal L(E,F;\Omega)$ and
$\Gamma = \Gamma (E,F;\Omega)$.
We begin by proving the lower bound.

\begin{proposition}\label{prop:lower-bound}
  Let $\Omega\subset X$ be nonempty, open and bounded
 and let $E,F\subset X$
 with $\overline{E}\cap \overline{F}=\emptyset$.
 Then $\Mod_p(\Gamma)< \infty$, and if also $\Mod_{\frac p{p-1}}(\mathcal{L})<\infty$,
 then
\[
1 \le \Mod_{\frac p{p-1}}(\mathcal{L})^{\frac{p-1}p}\Mod_p(\Gamma)^{\frac 1p}.
\]
\end{proposition}

\begin{proof}
	Since $\overline{E}\cap \overline{\Om}$ and
	$\overline{F}\cap \overline{\Om}$ are two disjoint compact sets,
	we have $d:=\dist(\overline{E}\cap \overline{\Om},\overline{F}\cap \overline{\Om})>0$ and so
	$\Mod_p(\Gamma)<\infty$;
	e.g. $d^{-1}\chi_{\Om}$ is an admissible function.
	By \cite[Proposition 2.17]{HK} and \cite[Theorem 1.11]{JJRRS} we have
	\begin{equation}\label{eq:capa agrees with mod}
	\rcapa_p(E,F;\Om)=\Mod_p(\Gamma);
	\end{equation}
	recall Definition \ref{def:condenser}.
	Then by \cite[Theorem 5.13]{BB-OD} we find a capacitary potential
	of $E$ and $F$ in $\Omega$, that is,
	a function $u\in N^{1,p}(\Om)$ such that $0\le u\le 1$ in $\Om$,
	$u=1$ in $E\cap\Om$, $u=0$ in $F\cap\Om$, and
	\[
	\rcapa_p(E,F;\Om)=\int_{\Om}g_u^p \,d\mu.
	\]
	We find two
	disjoint open sets $W_1,W_2\subset \Om$ with $\overline{E}\cap \Om\subset W_1$ and
	$\overline{F}\cap \Om\subset W_2$.
	Since $u$ is a capacitary potential,
	for any nonnegative
	$\varphi\in \Lip_c(W_1)\subset \Lip_c(\Om)$ we have that
	$u+\varphi$ is admissible for $\rcapa_p(E,F;\Om)$ and so
	\[
	\int_{\Om}g_u^p\,d\mu\le 
	\int_{\Om} g_{u+\varphi}^p\,d\mu,
	\]
	and so by the locality of minimal weak upper gradients
	(see e.g. \cite[Corollary 2.21]{BB}),
	\[
	\int_{\{\varphi\neq 0\}}g_u^p\,d\mu\le \int_{\{\varphi\neq 0\}} g_{u+\varphi}^p\,d\mu.
	\]
	Thus $u$ is superminimizer in $W_1$, and analogously a subminimizer in $W_2$.
	Let $u^*$ be the lsc-regularization of $u$ in $W_1$ and the analogously defined
	usc-regularization of $u$ in $W_2$, and $u^*=u$ in $\Om\setminus (W_1\cup W_2)$.
	Then by Proposition
	\ref{prop: facts about lsc regularization},
	$u^*$ is lower semicontinuous in $W_1$ and, analogously,
	upper semicontinuous in $W_2$, and
	$u=u^*$ $p$-q.e. in $\Om$.
	
	Let $\mathcal{L}'$ be the collection of super-level sets of $u^*$,
	$U_t:=\{x\in \Omega:u^*(x)>t\}$, for $t\in (0,1)$.
	By Theorem \ref{thm:fine continuity} we have $u^*=1$ in $b_p E\cap \Om$.
	Thus the sets $U_t\cap W_1$, for $t\in (0,1)$, are open and contain $b_p E\cap \Om$,
	and so each set $\inte (U_t)$ contains $b_p E\cap \Om$.
	Analogously,
	$b_p F\cap \Om\subset \ext(U_t)$ for all $t\in (0,1)$.
	In conclusion we have
	$\mathcal L'\subset \mathcal L$, or more precisely $P(U_t,\Om\cap\cdot)$
	is in $\mathcal L$ for every $t\in (0,1)$.
	Thus $\Mod_{\frac p{p-1}}(\mathcal L')\le \Mod_{\frac p{p-1}}(\mathcal L)<\infty$.
	
Let $\rho\in L^{p/(p-1)}(X)$ be any
admissible function for $\Mod_{\frac p{p-1}}(\mathcal{L}')$.
By e.g. \cite[Proposition 2.44]{BB} we know that $g_{u,1}\le g_u$ in $\Om$, where $g_{u,1}$ and $g_u$ are the minimal
$1$-weak and $p$-weak upper gradients, respectively, of $u$ in $\Om$.
Thus also $u\in N^{1,1}(\Om)$.
Since $\liploc(\Om)$ is dense in $N^{1,1}(\Om)$,
see \cite[Theorem 5.47]{BB}, it follows that $u\in\BV(\Om)$ with
$d\Vert Du\Vert\le g_{u,1}\,d\mu\le g_{u}\,d\mu$ in $\Om$.
Using also the coarea formula of
Theorem~\ref{thm:coarea}, we get
\begin{align*}
1 		&\le \int_0^1\left(\int_\Om \rho \,dP(U_t,\cdot) \right)dt \\
		&= \int_\Om \rho \,d\Vert Du\Vert\qquad\vert\vert\textrm{ since } u^*=u\ \textrm{a.e.} \\
		&\le \int_\Om \rho \,g_u \ d\mu \\
		&\le \left(\int_\Om \rho^{\frac p{p-1}} \,d\mu\right)^{\frac{p-1}p}\left(\int_\Om g_u^p \,d\mu\right)^{\frac 1p} \\
		&=\left(\int_\Om \rho^{\frac p{p-1}} \,d\mu\right)^{\frac{p-1}p}\Mod_p(\Gamma)^{\frac 1p},
\end{align*}
using also \eqref{eq:capa agrees with mod}.
Taking the infimum over admissible $\rho$, we get
\[
1 \le \Mod_{\frac p{p-1}}(\mathcal{L}')^{\frac{p-1}p}\Mod_p(\Gamma)^{\frac 1p} \le \Mod_{\frac p{p-1}}(\mathcal{L})^{\frac{p-1}p}\Mod_p(\Gamma)^{\frac 1p}.
\]
\end{proof}

In the case where $E$ and $F$ are compact, we get the lower bound also
for the following smaller family of surfaces:
 \begin{equation}\label{eq:smaller surface family}
 \mathcal L^*:=\mathcal L^*(E,F;\Om):=\{P(U,\Om\cap\cdot):\,
E \subset \inte(U) \text{ and } F\subset \ext (U)\}.
 \end{equation}
\begin{proposition}\label{prop:lower-bound with compact E and F}
  Let $\Omega\subset X$ be a nonempty bounded domain
 and let $E,F\subset \Om$ be disjoint nonempty compact sets.
 If 
 $\Mod_{\frac p{p-1}}(\mathcal{L^*})<\infty$,
 then
\[
1 \le \Mod_{\frac p{p-1}}(\mathcal{L^*})^{\frac{p-1}p}\Mod_p(\Gamma)^{\frac 1p}.
\]
\end{proposition}
\begin{proof}
The proof is almost the same as for Proposition \ref{prop:lower-bound};
we only need to note that since $E$ and $F$ are compact,
according to Theorem 1.1 in \cite{KaSh} we find
for every $\eps>0$ a function
$u\in \liploc(\Om)$ with $0\le u\le 1$ in $\Om$, $u=1$ in $E$, $u=0$ in $F$, and
	\[
	\int_{\Om}g_u^p \,d\mu<\rcapa_p(E,F;\Om)+\eps.
	\]
	Then we can consider the super-level sets
	$\{x\in \Omega:u(x)>t\}$ for $t\in (0,1)$, which all belong to $\mathcal L^*$.
\end{proof}

Now we prove the upper bound.
Part of the idea for the following proof came from Lohvansuu and Rajala \cite{LR};
the authors would like to thank them for sharing an early version of their manuscript.

\begin{proposition}\label{prop:upper-bound}
Let $\Omega\subset X$ be open and
$E,F\subset X$ be disjoint sets with
$\capa_p(b_p E\setminus E)=0$ and $\capa_p(b_p F\setminus F)=0$.
If $\Mod_{\frac p{p-1}} (\mathcal{L})=\infty$ then $\Mod_p(\Gamma)=0$, and if
$0<\Mod_{\frac p{p-1}} (\mathcal{L})<\infty$
then
\[
\Mod_{\frac p{p-1}} (\mathcal{L})^{\frac{p-1}p}
\Mod_p(\Gamma)^{\frac 1p} \le C
\]
for a constant $C$.
\end{proposition}

In particular, $E$ and $F$ can be closed sets.

\begin{proof}
Let $x\in X \setminus b_p E$.
Since $\capa_p(b_p E\setminus E)=0$, by definition of the variational capacity we get
\[
\rcapa_p(B(x,t)\cap b_p E, B(x,2t))\le \rcapa_p(B(x,t)\cap E, B(x,2t))
\]
for every
$r>0$. Thus
\begin{align*}
\int_0^1 &\left(\frac{\rcapa_p(B(x,t)\cap b_p E, B(x,2t))}{\rcapa_p(B(x,t), B(x,2t))} \right)^{\frac{1}{p-1}} \frac{dt}{t}\\
& \hspace{1.4in}\le \int_0^1 \left(\frac{\rcapa_p(B(x,t)\cap E, B(x,2t))}{\rcapa_p(B(x,t), B(x,2t))} \right)^{\frac{1}{p-1}} \frac{dt}{t} <\infty.
\end{align*}
Thus, $X\setminus b_p E$ is $p$-finely open. Similarly,
$X\setminus b_p F$ is $p$-finely open and so $X\setminus (b_p E\cup b_p F)$ is as well.

Now by Theorem~1.4 in \cite{BBL},
we have that $X\setminus (b_p E\cup b_p F)$ is $p$-quasiopen.
Then for each $i\in \N$, we can find an open set $G_i\subset X$ 
with $\capa_p(G_i)<1/i$ and so that $(X\setminus (b_p E \cup b_p F))\cup G_i$ is open.
We can assume that the sets $G_i$ form a decreasing sequence.
Furthermore, we know that $\capa_p(E\setminus b_p E)=0$ and $\capa_p(F\setminus b_p F)=0$
(see Corollary~1.3 in \cite{BBL}) so then since $\capa_p$ is an outer capacity,
we can choose $G_i$ to contain
$E\Delta b_p E$
and $F\Delta b_p F$, where $\Delta$ denotes the symmetric difference.
We now have that
$(b_p E\cup b_p F)\setminus G_i=(E\cup F)\setminus G_i$ and this is a closed set. 

Take open sets $\Om_1\Subset \Om_2\Subset \ldots \Subset \Om$ with
$\Om=\bigcup_{i=1}^{\infty}\Om_i$.
Define
\[
\Gamma_i :=\{\gamma \in \Gamma(E,F;\Omega) :\, |\gamma| \subset \Om_i\setminus G_i\},
\]
where $|\gamma|$ is the image of $\gamma$ in $X$.
Fix $i\in\N$ and a
rectifiable curve $\gamma \in\Gamma_i$ (assume for now
that $\Gamma_i\neq \emptyset$).
Let
\[
\mathcal L_{i,j}:=\{P(U,\Om\cap\cdot):\,
\dist(\overline{\Om_i}\cap E\setminus G_i,X\setminus U)>j^{-1}
\text{ and }
\dist(\overline{\Om_i}\cap F\setminus G_i,U)>j^{-1}\}.
\]
Also fix $j\in\N$.
We wish to construct an admissible function for $\mathcal L_{i,j}$.
First we construct a Whitney covering of $\gamma$. 
Set
\[
d(x):=\dist(x,(\overline{\Om_i}\cap(E \cup  F)\setminus G_i))\cup (X\setminus \Om))
\]
and note that $d(x)>0$ for all $x\in |\gamma|\setminus (E\cup F)$.
For $k\in\Z$ set 
\[
|\gamma|_k := \{x\in|\gamma| :\, 2^{k-1}<d(x)\le 2^k\}
\]
and
\[
\mathcal{F}_k:=\left\{B\Big(x,\frac{d(x)}{25\lambda}\Big) :\, x\in|\gamma|_k\right\}.
\]
So $\mathcal{F}_k$ forms a cover of $|\gamma|_k$.
Then by the $5$-covering theorem, we can find a pairwise disjoint subcollection $\mathcal{G}_k\subset \mathcal{F}_k$ such that 
\[
|\gamma|_k\subset\bigcup_{B\in\mathcal{F}_k} B \subset \bigcup_{B\in\mathcal{G}_k} 5B.
\]
Since $|\gamma|_k$ is bounded, $\mathcal G_k$ is finite for each $k$. 
Letting $\mathcal B := \bigcup_{k\in\Z} \mathcal G_k$, the collection of five times enlarged balls from $\mathcal B$ is a cover for $|\gamma|\setminus (E\cup F)$.
Now for $U\in\mathcal L_{i,j}$, set
\[
T:=\sup\left\{t\in(0,\ell_{\gamma}) :\, \frac{\mu(U\cap 5B)}{\mu(5B)} \ge \frac 12 \text{ for all } B \in\mathcal B \text{ such that } \gamma(t) \in 5B\right\}.
\]
We know that the above supremum is attained and $T \in (0,\ell_{\gamma})$ since
$\overline{\Om_i}\cap b_p E \setminus G_i=\overline{\Om_i}\cap E\setminus G_i$
is a compact subset of the open set $\inte(U)$
and
$\overline{\Om_i}\cap b_p F \setminus G_i=\overline{\Om_i}\cap F\setminus G_i$
is a compact subset of the open set $\ext(U)$, and by the definition of $\mathcal B$.
Let $B_1\in\mathcal B$ such that $\gamma(T)\in 5B_1$.
By continuity of $\gamma$ and definition of $T$, there exists
$\kappa>0$ such that $\gamma(T-\kappa)\in 5B_1$ and
$\frac{\mu(U\cap 5B)}{\mu(5B)} \ge \frac 12$ for all
$B \in\mathcal B$ such that $\gamma(T-\kappa) \in 5B$. Thus
$\frac{\mu(U\cap 5B_1)}{\mu(5B_1)}\ge \frac 12$.
Again by continuity of $\gamma$, there exists $\delta >0$ such that $\gamma(T+\delta)\in 5B_1$.
Then since $T +\delta > T$, we know that there exists a ball $B_2\in\mathcal B$
with $\gamma(T+\delta)\in 5B_2$ such that $\frac{\mu(U\cap 5B_2)}{\mu(5B_2)} < \frac 12$. 
So we have
\begin{equation}\label{eq-1}
\frac{\mu(U\cap 5B_2)}{\mu(5B_2)} < \frac 12\le\frac{
\mu(U\cap 5B_1)}{\mu(5B_1)}.
\end{equation}
By the fact that $5B_1\cap 5B_2$ is nonempty
(since it contains $\gamma(T+\delta)$), it is easy to check that
$\rad(B_1)\le 2\rad(B_2)\le 4\rad(B_1)$.
Hence $5B_2\subset 25 B_1$.
By using first \eqref{eq-1} and the doubling property and
then the relative isoperimetric inequality of
Definition \ref{def:relative isoperimetric inequality}, we get
for some constant $\widetilde C>0$ (depending only on the doubling constant)
\begin{equation}\label{eq:rel-isoperim}
\frac{1}{\widetilde C} 
\le 
\min\left\{\frac{\mu(25B_1\cap U)}{\mu(25B_1)},\frac{\mu(25B_1\setminus U)}{\mu(25B_1)}\right\}
\le \frac{25C_I \rad (B_1)}{\mu(B_1)} P(U, 25\lambda B_1).
\end{equation}
Choose $K_j\in\Z$ so that
$2^{K_j}<\frac{1}{2}\min\left\{\frac{1}{j},\dist(\Om_i',X\setminus \Om)\right\}$. 
Then for any $k\le K_j$ and $B\in\mathcal G_k$, either
$25B \subset U$ or $25B$ and $U$ are disjoint, which implies that
$\frac{\mu(U\cap 25B)}{\mu(25B)}\in\{0,1\}$.
Therefore we know that $B_1 \in \mathcal G_{k}$
for some $k\ge K_j$.
Define
\[
\phi_{i,j} := 25 C_I\widetilde{C}\sum_{k\ge K_j} \sum_{ B\in\mathcal G_k} \frac{\rad(B)}{\mu(B)} 
\chi_{25\lambda B}.
\]
Recall that $\mathcal G_k$ is finite for each $k$. 
Also $|\gamma|$ is bounded, so there exists a $K_0\in\Z$ such that $\mathcal G_k$ is empty for all $k\ge K_0$. 
Hence the function $\phi_{i,j}$ is $p/(p-1)$-integrable.
Furthermore, $\phi_{i,j}$ is admissible for $\mathcal{L}_{i,j}$, since for any
$U\in \mathcal L_{i,j}$, by \eqref{eq:rel-isoperim} we have
\begin{align*}
\int_{\Omega}\phi_{i,j}\,dP(U,\cdot) 
& = 25 C_I\widetilde{C}\int_{\Omega} \sum_{k\ge K_j}
\sum_{B\in\mathcal G_k} \frac{ \rad(B)}{\mu(B)} \chi_{25\lambda B} \,dP(U,\cdot)\\
& = 25 C_I\widetilde{C}\sum_{k\ge K_j}
\sum_{B\in\mathcal G_k} \frac{\rad(B)}{\mu(B)} P(U,25\lambda B)\\
&\ge 1.
\end{align*}
Using Lemma \ref{lem:existence of minimal rho}, pick
a $p/(p-1)$-weakly admissible function $\rho_{i,j}$ such that 
\[
\Mod_{\frac p{p-1}}(\mathcal L_{i,j}) = \int_X \rho_{i,j}^{\frac p{p-1}} \,d\mu.
\]
Recall the definition of the noncentered Hardy-Littlewood maximal function from
\eqref{eq:definition of maximal function}.
We now apply Lemma \ref{lem:phi-and-rho} which gives
\begin{equation}\label{eq:admissibility for Gamma i}
\begin{split}
\Mod_{\frac p{p-1}}(\mathcal L_{i,j}) 
& \le \int_X \phi_{i,j}\, \rho_{i,j}^{\frac 1{p-1}}\, d\mu \\
& \le C \int_X \sum_{k\ge K_j} \sum_{B \in \mathcal{G}_k} \frac{\rad(B)}{\mu(B)} \chi_{25\lambda B} \;\rho_{i,j}^{\frac 1{p-1}}\, d\mu \\
&\le C \sum_{k\in\Z} \sum_{B \in \mathcal{G}_k} \frac{\rad(B)}{\mu(B)} \int_{25\lambda B} \rho_{i,j}^{\frac 1{p-1}}\, d\mu \\
& \le C \sum_{k\in\Z} \sum_{B \in \mathcal{G}_k}  \rad(B) \avgint_{25\lambda B} \rho_{i,j}^{\frac 1{p-1}} d\mu \\
&  \le C \sum_{k\in\Z} \sum_{B \in \mathcal{G}_k}  \rad(B) \inf_{x\in B} \mathcal{M}\rho_{i,j}^{\frac 1{p-1}}(x) \\
& \le C \int_{\gamma} \mathcal{M}\rho_{i,j}^{\frac 1{p-1}} \,ds;
\end{split}
\end{equation}
the last inequality holds because the curve $\gamma$ travels at least the length
$\rad(B)$ inside $B$,
and the balls in each $\mathcal G_k$ are pairwise disjoint
and clearly two balls $B_1\in \mathcal G_k$ and $B_2\in \mathcal G_l$ can only intersect
if $|k-l|=1$.

Now we show that $\bigcup_j \mathcal L_{i,j} = \mathcal L$.
First note that
\[
\bigcup_j \mathcal L_{i,j} =\{U\in\mathcal L :\,
\dist(\overline{\Om_i}\cap E\setminus G_i,X\setminus U)>0
\text{ and }
\dist(\overline{\Om_i}\cap F\setminus G_i,U)>0\}.
\]
Now if we take $U\in\mathcal L$, then $\overline{\Omega_i}\cap b_p E\setminus G_i \subset \Omega \cap b_p E \subset \inte (U)$, i.e. $\overline{\Omega_i} \cap b_p E\setminus G_i$ is a compact subset of the open set $\inte (U)$. 
Hence there is a strictly positive distance between
$\overline{\Omega_i}\cap b_p E \setminus G_i=\overline{\Omega_i}\cap E \setminus G_i$
and $X\setminus U$.
A similar argument shows that there is a strictly positive distance between
$\overline{\Omega_i} \cap F \setminus G_i$ and $U$.
Therefore $U\in \bigcup_j \mathcal L_{i,j}$, proving that
$\bigcup_j \mathcal L_{i,j} = \mathcal L$.

Now note that since we are assuming $\Mod_{\frac p{p-1}}(\mathcal L)\neq 0$,
and the families $\mathcal{L}_{i,j}$ are increasing in $j$,
by Lemma \ref{lem:mod-continuity}
we have $\Mod_{\frac p{p-1}}(\mathcal L_{i,j})>0$
for all sufficiently large $j\in\N$
(with $i$ still fixed).
Thus by \eqref{eq:admissibility for Gamma i},
$C\Mod_{\frac p{p-1}}(\mathcal L_{i,j})^{-1}\mathcal M \rho_{i,j}^{\frac 1{p-1}}$
is admissible for $\Gamma_i$.
Therefore
\[
\Mod_p(\Gamma_i) 
\le C\Mod_{\frac p{p-1}}(\mathcal L_{i,j})^{-p} \int_X \left(\mathcal M \rho_{i,j}^{\frac 1{p-1}}\right)^p d\mu.
\]
Since the maximal function is a bounded operator from $L^p(X)$ to $L^p(X)$ when
$1<p<\infty$, see e.g. \cite[Theorem 3.13]{BB}, we get
\begin{align*}
\Mod_p (\Gamma_i)^{\frac 1p} 
& \le C \Mod_{\frac{p}{p-1}} (\mathcal L_{i,j})^{-1} \Big\Vert \mathcal M \rho_{i,j}^{\frac 1{p-1}} \Big\Vert_{L^p(X)} \nonumber \\ 
& \le C \Mod_{\frac{p}{p-1}} (\mathcal L_{i,j})^{-1} \Big\Vert \rho_{i,j}^{\frac 1{p-1}} \Big\Vert_{L^p(X)} \nonumber \\
& = C \Mod_{\frac{p}{p-1}} (\mathcal L_{i,j})^{-1} \left(\int_X \rho_{i,j}^{\frac p{p-1}}\,d\mu \right)^{\frac 1p} \nonumber \\
& = C \Mod_{\frac{p}{p-1}} (\mathcal L_{i,j})^{\frac{1-p}{p}}. 
\end{align*}
Recalling that
$\lim_{j\to \infty} \Mod_{\frac{p}{p-1}} (\mathcal L_{i,j}) 
= \Mod_{\frac p {p-1}} (\mathcal L)$,
we get
\begin{equation}\label{eq:Gamma-L}
\Mod_p(\Gamma_i)^{\frac 1p}
\le C \Mod_{\frac p{p-1}}(\mathcal L)^{\frac{1-p}p},
\end{equation}
and in particular $\Mod_p(\Gamma_i)=0$ if
$\Mod_{\frac p{p-1}}(\mathcal L)=\infty$.
Note that \eqref{eq:Gamma-L} holds also if $\Gamma_i=\emptyset$.

Finally note that the sequence $\Gamma_i$ is increasing with
$\bigcup_{i\in\N} \Gamma_i=\Gamma\setminus N$, where 
\[
N:=\left\{\gamma\in\Gamma :\, |\gamma|\cap \bigcap_i G_i \neq \emptyset\right\}.
\] 
But $\capa_p\left(\bigcap_i G_i\right)=0$, and so $\Mod_p(N)=0$,
see e.g. \cite[Proposition 1.48]{BB}.
So by Lemma \ref{lem:mod-continuity} we have
$\lim_{i\to\infty} \Mod_p(\Gamma_i)=\Mod_p(\Gamma\setminus N)=\Mod_p(\Gamma)$.
Combining this with \eqref{eq:Gamma-L} above, we get 
\[
\Mod_p(\Gamma)^{\frac 1p}=\lim_{i\to\infty}\Mod_p(\Gamma_i)^{\frac 1p}
\le C\Mod_{\frac p{p-1}}(\mathcal L)^{\frac{1-p}p},
\]
and in particular $\Mod_p(\Gamma)=0$ if
$\Mod_{\frac p{p-1}}(\mathcal L)=\infty$.
\end{proof}

Now Theorem \ref{thm:duality of moduli} from the introduction follows from
Proposition \ref{prop:lower-bound} and Proposition \ref{prop:upper-bound}.
We give the theorem in the following more precise form.
\begin{theorem}\label{thm:duality of moduli in text}
Let $\Om\subset X$ be nonempty, open and bounded and let
$E,F\subset X$ be such that $\overline{E}\cap \overline{F}=\emptyset$ and
$\capa_p(b_p E\setminus E)=0$
and $\capa_p(b_p F\setminus F)=0$.
If $\Mod_{\frac p{p-1}} (\mathcal{L})=\infty$ then $\Mod_p(\Gamma)=0$, and else
\[
1 \le \Mod_{\frac p{p-1}}(\mathcal{L})^{\frac{p-1}p}\Mod_p(\Gamma)^{\frac 1p}\le C
\]
for some constant $C\ge 1$.
\end{theorem}

\begin{proof}
If $\Mod_{\frac p{p-1}} (\mathcal{L})=\infty$ then $\Mod_p(\Gamma)=0$ by
Proposition \ref{prop:upper-bound}. Else
Proposition \ref{prop:lower-bound} gives $\Mod_p(\Gamma)<\infty$
as well as the lower bound of the theorem, and
in particular this guarantees $0<\Mod_p(\Gamma)<\infty$. Then the upper bound
follows from Proposition \ref{prop:upper-bound}.
\end{proof}

\section{Proof of Theorem \ref{thm:quasiconformality}}\label{sec:qc characterization}

Now we can prove Theorem \ref{thm:quasiconformality} given in the introduction.
We give it in the following somewhat more general form.

\begin{theorem}\label{thm:quasiconformality in text}
Suppose that $(Y,d_Y,\mu_Y)$ is another complete metric space
that supports a $1$-Poincar\'e inequality,
such that $\mu_Y$ is doubling and
\begin{equation}\label{eq:mass bounds for X and Y}
\mu(B(x,r))\ge C_0^{-1}r^p\quad\textrm{and}\quad\mu_Y(B(y,r))\le C_0 r^p
\end{equation}
for all $x\in X$, $y\in Y$, $r>0$, and a constant $C_0> 0$.
Suppose $f\colon X\to Y$ is a homeomorphism such that for every collection
of surfaces $\mathcal L=\mathcal L(E,F;\Om)$ with $\Om\subset X$
nonempty, open and bounded and $E,F\subset \Om$ compact,
we have
\[
\Mod_{\frac p{p-1}}(f\mathcal L)\le C_0 \Mod_{\frac p{p-1}}(\mathcal L),
\]
where
\[
f\mathcal L=\{P(f(U),f(\Om)\cap\cdot):\,
b_p E\cap \Om\subset \inte(U) \text{ and } b_p F\cap \Om\subset \ext (U)\}.
\]
Then $f$ is quasiconformal with a constant depending only on $C_0$, $p$,
and the space $X$.
\end{theorem}

Of course, \eqref{eq:mass bounds for X and Y} is satisfied in particular
if $X$ and $Y$ are both
Ahlfors $p$-regular; recall Definition \ref{def:doubling and Ahlfors regularity}.
Also recall the definitions of $L_f$ and $l_f$ from Definition~\ref{def:dilations}.

\begin{proof}
As complete metric spaces equipped with a doubling measure and supporting a Poincar\'e
inequality, $X$ and $Y$ are \emph{quasiconvex}, see e.g. \cite[Theorem 4.32]{BB},
and so for each of them
a biLipschitz change in the metric gives a geodesic space
(see Section 4.7 in \cite{BB}).
Since the theorem is easily seen to be invariant under biLipschitz
changes in the metrics, we can assume that $X$ and $Y$ are geodesic.

We want to apply Proposition~\ref{prop:lower-bound with compact E and F}
and Proposition~\ref{prop:upper-bound} to suitable sets defined via the homeomorphism
$f$. Fix $x\in X$ and $r>0$ and let $L:=L_f(x,r)$ and $l:=l_f(x,r)>0$.
Suppose also that $L>2C_1 l$, where $C_1$ is the constant from
Theorem \ref{thm:linear local connectivity}.
By choosing $r$ sufficiently small, we have $L<\diam (Y)/4$ (we can assume that
$\diam(Y)>0$).
Since $f(\overline{B}(x,r))$ is compact,
there exists $y\in f(\overline{B}(x,r))$ such that $d_Y(f(x),y)=L$.

Let $E:=\fin(\overline{B}(f(x),l))$,
and $F:=f^{-1}(F_*)$ where $F_*$ is the maximal connected set
containing $y$ and contained in
$\overline{B}(f(x),M)\setminus B(f(x),L/C_1)$, for some fixed $M\ge 2C_1 L$.
By Theorem \ref{thm:linear local connectivity}
(note that here we use the upper bound in \eqref{eq:mass bounds for X and Y})
we have
\begin{equation}\label{eq:Fstar contains}
F_*\supset B(f(x),M/C_1)\setminus B(f(x),L).
\end{equation}
Note that balls are connected in geodesic spaces,
and $f$ is a homeomorphism, so $E$ and $F$ are connected.
Both $E$ and $F$ are moreover closed, and since
$X$ and $Y$ are proper, $f$ and $f^{-1}$ map bounded sets to bounded sets,
and so $E$ and $F$ are also bounded and thus compact.
Since $Y$ is connected, the set $F_*$ and thus also the set $F$ consists of at
least 2 points and so $\diam(F)>0$.
If $r\to 0$ then $\diam(E)\to 0$,
and thus by choosing $r$ even smaller if necessary,
we can assume that $\diam (E)$ is less than $\diam (F)$.
Note that $\Om:=f^{-1}(B(f(x),M+1))$ is also bounded.
Let $\Gamma=\Gamma(E,F;\Omega)$ and $\mathcal{L}=\mathcal L(E,F;\Omega)$.

For the family $f\Gamma$ consisting of the 
curves $f\circ\gamma$, with $\gamma\in\Gamma$, we have
\[
f\Gamma=\Gamma(\overline{B}(f(x),l),F_*;B(f(x),M+1)).
\]
From this it is easy to see that every curve in $\Gamma(E,F;X)$ has a subcurve in $\Gamma$,
and so $\Mod_p(\Gamma)=\Mod_p(\Gamma(E,F;X))$; see e.g. \cite[Lemma 1.34(c)]{BB}.
Notice that $\fin (y)\in F\cap \overline{B}(x,r)$,
and we know that $x\in E$, so $\dist(E,F)\le r$.
It is straightforward to show that there is some $z\in X\setminus B(x,r)$
with $d_Y(f(x),f(z))=l$.
Thus $r\le \diam(E)$, which we noted to be less than
$\diam (F)$, and so
\[
\frac{\dist (E,F)}{\min\{\diam(E),\diam(F)\}}\le \frac rr = 1.
\]
By Theorem \ref{thm:Loewner} we know that $X$ is a Loewner space (note that
here we need the lower mass bound in \eqref{eq:mass bounds for X and Y}), and so
\begin{equation}\label{eq:Gamma-phi}
\Mod_p(\Gamma)=\Mod_p(\Gamma(E,F;X))\ge \phi(1)>0,
\end{equation}
where $\phi$ is the Loewner function for $X$.
We observe that every curve in $f\Gamma$ has a subcurve in the family
$\Gamma(\overline{B}(f(x),l),Y\setminus B(f(x),L/C_1);Y)$.
Thus
\[
\Mod_p(f\Gamma)\le \Mod_p(\Gamma(\overline{B}(f(x),l),Y\setminus B(f(x),L/C_1));Y).
\]
Now by Proposition 5.3.9 in \cite{HKST15} we have
\begin{equation}\label{eq:Mod of f Gamma}
\Mod_p(f\Gamma)\le C_2\left(\log\dfrac {L}{C_1 l} \right)^{1-p}
\end{equation}
for a constant $C_2$ depending only
on $C_0$ and $p$
(here we need the upper mass bound
in \eqref{eq:mass bounds for X and Y}).

Since $L<\diam(Y)/4$ and $Y$ is connected, by \eqref{eq:Fstar contains}
there exists a ball $B\subset F_*$ with $\rad(B)=L/2$.
Then from the relative isoperimetric inequality of
Definition \ref{def:relative isoperimetric inequality} and the doubling property
of $\mu_Y$ we see that for every $U\in f\mathcal L$,
\[
P(U,B(f(x), 2\lambda M))\ge 
 (2C_I M)^{-1}\min \{\mu_Y(B(f(x),l)),\mu_Y(B)\}=:c>0.
\]
It follows that $\Mod_{p/(p-1)}(f\mathcal L)<\infty$, as
$c^{-1}\chi_{B(f(x),2 \lambda M)}$ is an admissible test function.
Recall the definition of the family $\mathcal L^*\subset\mathcal L$
from \eqref{eq:smaller surface family}.
For this family it is easy to verify (since $f$ is a homeomorphism) that
\[
f\mathcal L^*=\mathcal L^*(\overline{B}(f(x),l),F^*;B(f(x),M+1)).
\]
Thus by Proposition~\ref{prop:lower-bound with compact E and F}
and the assumption of the quasi-preservation of modulus of surfaces,
\[
1 \le \Mod_{\frac p{p-1}}(f\mathcal L^*)^{\frac{p-1}p} \Mod_p(f\Gamma)^{\frac 1p}
\le C_0^{\frac{p-1}p} \Mod_{\frac p{p-1}}(\mathcal L)^{\frac{p-1}p} \Mod_p(f\Gamma)^{\frac 1p}.
\]
Therefore, by Proposition~\ref{prop:upper-bound} and since we had $\Mod_p(\Gamma)\neq 0$
(recall \eqref{eq:Gamma-phi})
\[
\Mod_p(\Gamma)^{\frac 1p} \le C_0^{\frac{p-1}p} \Mod_p(\Gamma)^{\frac 1p} \Mod_{\frac p{p-1}}(\mathcal L)^{\frac{p-1}p} \Mod_p(f\Gamma)^{\frac 1p} \le C_0^{\frac{p-1}p} C \Mod_p(f\Gamma)^{\frac 1p}. 
\]
Combining this with \eqref{eq:Gamma-phi} and \eqref{eq:Mod of f Gamma}, we have
\[
\phi(1)\le \Mod_p(\Gamma)\le C_0^{p-1} C^p C_2 \left(\log\frac {L}{C_1 l}\right)^{1-p}.
\]
Thus
\[
\frac{L}{l}\le C_1\exp{\left(C_0 C^\frac p{p-1} 
C_2^\frac 1{p-1}\phi(1)^{\frac 1{1-p}}\right)}.
\]
Recall that we were assuming $L>2C_1l$; in conclusion
\[
\limsup_{r\to 0^+} \frac {L_f(x,r)}{l_f(x,r)} 
\le C_1\max\left\{2,\exp{\left(C_0 C^\frac p{p-1} 
C_2^\frac 1{p-1}\phi(1)^{\frac 1{1-p}}\right)}\right\}
\]
for every $x\in X$.
Therefore $f$ is quasiconformal.
\end{proof}

\begin{remark}
Note that in the above proof we employed the family $\mathcal L^*$ because
it is not
clear that
\[
f \mathcal L=\mathcal L(\overline{B}(f(x),l),F_*;B(f(x),M+1)).
\]
This is the case because it is not clear that
\[
b_p E\subset \inte(f^{-1}(U))\quad\textrm{and}\quad
b_p F\subset \ext (f^{-1}(U))
\]
for every $U\in \mathcal L(\overline{B}(f(x),l),F_*;B(f(x),M+1))$,
as would be required in the definition of
$\mathcal L(E,F;\Om)$.
In other words, the image under $f$ or $f^{-1}$
of every ``separating surface'' 
might not be a ``separating surface''.
It is known, at least in Ahlfors regular spaces,
that a quasiconformal mapping
(whose inverse is also quasiconformal)
preserves the \emph{measure-theoretic}
interior, exterior, and boundary, see \cite{GK}, \cite[Theorem 6.1]{KMS},
and \cite[Lemma 4.8]{JLS}.
If we knew a similar property to hold for
capacitary thickness points, then the above problem would not arise.
Thus we ask:
\begin{itemize}
\item If $f\colon X\to Y$ is a quasiconformal mapping, do we have
$f(b_p E)=b_p f(E)$ for every (closed) set $E\subset X$?
\end{itemize}
\end{remark}

\noindent Address:

\vskip1ex plus 1ex minus 0.5ex
\noindent R.J.: Department of Mathematical Sciences,\\
P. O. Box 210025,
University of Cincinnati,
Cincinnati, OH 45221-0025,
U.S.A.

\vskip1ex plus 1ex minus 0.5ex
\noindent E-mail:
{\tt jones3rh@mail.uc.edu}
\vskip1ex plus 1ex minus 0.5ex
\noindent P.L.: Institut f\"ur Mathematik,
Universit\"at Augsburg,\\
Universit\"atsstr. 14, 86159 Augsburg, Germany
\vskip1ex plus 1ex minus 0.5ex
\noindent E-mail: {\tt panu.lahti@math.uni-augsburg.de}

\end{document}